\documentclass[]{amsart}
\usepackage{amscd,amsthm,amssymb,amsfonts,amsmath,euscript}

\theoremstyle{plain}
\newtheorem{thm}{Theorem}
\newtheorem{lemma}[thm]{Lemma}
\newtheorem{prop}[thm]{Proposition}

\theoremstyle{definition}
\newtheorem{defn}[thm]{Definition}

\theoremstyle{remark}
\newtheorem{remark}[thm]{Remark}

\newcommand{\nc}{\newcommand}

\numberwithin{equation}{section}

\def\makeop#1{\expandafter\def\csname#1\endcsname
  {\mathop{\rm #1}\nolimits}\ignorespaces}
\makeop{Hom}   \makeop{End}   \makeop{Aut}   \makeop{Isom}  \makeop{Pic} 
\makeop{Gal}   \makeop{ord}   \makeop{Char}  \makeop{Div}   \makeop{Lie} 
\makeop{PGL}   \makeop{Corr}  \makeop{PSL}   \makeop{sgn}   \makeop{Spf}
\makeop{Spec}  \makeop{Tr}    \makeop{Nr}    \makeop{Fr}    \makeop{disc}
\makeop{Proj}  \makeop{supp}  \makeop{ker}   \makeop{im}    \makeop{dom}
\makeop{coker} \makeop{Stab}  \makeop{SO}    \makeop{SL}    \makeop{SL}
\makeop{Cl}    \makeop{cond}  \makeop{Br}    \makeop{inv}   \makeop{rank}
\makeop{id}    \makeop{Fil}   \makeop{Frac}  \makeop{GL}    \makeop{SU}
\makeop{Nrd}   \makeop{Sp}    \makeop{Tr}    \makeop{Trd}   \makeop{diag}
\makeop{Res}   \makeop{ind}   \makeop{depth} \makeop{Tr}    \makeop{st}
\makeop{Ad}    \makeop{Int}   \makeop{tr}    \makeop{Sym}   \makeop{can}
\makeop{length}\makeop{SO}    \makeop{torsion} \makeop{GSp} \makeop{Ker}
\makeop{Adm}
\def\makebb#1{\expandafter\def
  \csname bb#1\endcsname{{\mathbb{#1}}}\ignorespaces}
\def\makebf#1{\expandafter\def\csname bf#1\endcsname{{\bf
      #1}}\ignorespaces} 
\def\makegr#1{\expandafter\def
  \csname gr#1\endcsname{{\mathfrak{#1}}}\ignorespaces}
\def\makescr#1{\expandafter\def
  \csname scr#1\endcsname{{\EuScript{#1}}}\ignorespaces}
\def\makecal#1{\expandafter\def\csname cal#1\endcsname{{\mathcal
      #1}}\ignorespaces} 

\def\doLetters#1{#1A #1B #1C #1D #1E #1F #1G #1H #1I #1J #1K #1L #1M
                 #1N #1O #1P #1Q #1R #1S #1T #1U #1V #1W #1X #1Y #1Z}
\def\doletters#1{#1a #1b #1c #1d #1e #1f #1g #1h #1i #1j #1k #1l #1m
                 #1n #1o #1p #1q #1r #1s #1t #1u #1v #1w #1x #1y #1z}
\doLetters\makebb   \doLetters\makecal  \doLetters\makebf
\doLetters\makescr 
\doletters\makebf   \doLetters\makegr   \doletters\makegr
     \def\qed{\qedmark\medbreak}%
\def\qedmark{{\enspace\vrule height 6pt width 5pt depth 1.5pt}}%

\normalsize

\makeop{Bl}

\def\Qp{{\bbQ}_p}

\newcommand{\npr}{\noindent }

\nc{\embed}{\hookrightarrow}

\newcommand{\ch}{characteristic }

\nc{\ol}{\overline}
\nc{\wt}{\widetilde}
\nc{\opp}{\mathrm{opp}}

\makeop{Ram}
\makeop{Rep}

\begin{document}
\renewcommand{\thefootnote}{\fnsymbol{footnote}}
\setcounter{footnote}{-1}


\title{Characteristic polynomials of central simple algebras}
\author{Chia-Fu Yu}
\address{
Institute of Mathematics, Academia Sinica and NCTS (Taipei Office)\\
Astronomy Mathematics Building \\
No. 1, Roosevelt Rd. Sec. 4 \\ 
Taipei, Taiwan, 10617} 
\email{chiafu@math.sinica.edu.tw}


\date{\today. 
}                      

\begin{abstract}
  We characterize characteristic polynomials of elements in a central
  simple algebra. We also give an account for the theory of rational
  canonical forms for separable linear transformations 
  over a central division algebra, and a description of
  separable conjugacy classes of the multiplicative group. 
\end{abstract} 

\maketitle




\def\Mat{{\rm Mat}}
\def\c{{\rm c}}
\def\i{{\rm i}}

\section{Introduction}
\label{sec:01}



Consider a $p$-adic local field $F$ (i.e. a finite field extension of
$\Qp$) and the connected reductive groups $G$ and $G'$ over $F$ of
multiplicative groups of the matrix algebra $\Mat_n(F)$ and a central
division algebra $D$ over $F$ of degree $n$. The local
Jacquet-Langlands correspondence states that there is a one-to-one
correspondence
\[ E^2(G') \simeq E^2(G') \]
between the sets $E^2(G')$ and $E^2(G)$ of essentially
square-integral irreducible smooth representations of $G'$ and $G$
\cite[p. 34]{deligne-kazhdan-vigneras}.
The correspondence is characterized by the following character identity
\begin{equation}
  \label{eq:1}
  \chi_{\pi'}(g')=(-1)^{n-1}\chi_{\pi}(g), \quad \forall\, \{g\}
  \text{ corresponds to} \{g'\} 
\end{equation}
if $\pi$ corresponds to $\pi'$, where $\chi_{\pi'}$ (resp. $\chi_\pi$)
is the character of $\pi'$ (resp. $\pi$), which is a locally constant 
central function defined
at least on the open set $G'(F)^{\rm reg}$ 
(resp. $G(F)^{\rm reg}$) of regular semi-simple elements.
The conjugacy class $\{g\}$ corresponding to
the conjugacy class $\{g'\}$ means that they share the same
characteristic polynomial. Clearly not every regular semi-simple
conjugacy class in $G(F)$ corresponds to a class in $G'(F)$. In this
Note we consider the following basic question in full generality:  \\

\npr {\bf (Q)} Let $A$ be a finite-dimensional central simple algebra
over an arbitrary
field $F$. Which polynomials of degree $\deg(A)$ 
are the characteristic polynomials of elements in $A$? \\

We recall some basic definitions for central simple algebras; 
see \cite{reiner:mo}.

\begin{defn}
  Let $A$ be a (f.d.) central simple algebra over a field
  $F$. 
\begin{enumerate}
\item The {\it degree}, {\it capacity}, and {\it index} of $A$
are defined as
\[ \deg(A):=\sqrt{[A:F]},\quad \c(A):=n,\quad \i(A):=\sqrt{[\Delta:F]}, \]
respectively 
if $A\cong \Mat_n(\Delta)$, where $\Delta$ is a division algebra over
$F$, which is uniquely determined by $A$ up to isomorphism.
The algebra $\Delta$ is also called the {\it division part} of $A$.
\item For any element $x\in A=\Mat_n(\Delta)$, the {\it characteristic
    polynomial of $x$} is defined to be the characteristic polynomial of
  the image of $x$ in $\Mat_{nd}(\bar F)$ under a map
\[ A\to A\otimes_F {\bar F} \stackrel{\rho}{\simeq} \Mat_{nd}(\bar
F), \]
where $\bar F$ is an algebraic closure of $F$ and $d$ is the degree of
$\Delta$. 
This polynomial is independent of
the choice of the isomorphism $\rho$ and it is defined over $F$. 
Denote by $f_x(t)$ the \ch polynomial of $x$. 
\end{enumerate}
\end{defn}

The question {\bf (Q)} would become more interesting when one
specializes $F$ as a global field field.
When $A$ is a quaternion algebra over a number field $F$, the answer  
is well-known. It
is easier to treat the case when the polynomial $f(t):=t^2+at+b$
factorizes in $F[t]$ (the answer is yes if and only if $f(t)$ has
double roots or $A$ splits). 
So we may consider the case when $f(t)$ is irreducible. Let $K$
be a splitting field of $f(t)$, which is unique up to isomorphism. 
Then $f(t)$ is a characteristic polynomial of $A$ 
if and only if there is a $F$-algebra embedding of $K$ in $A$. 
Then one can check the latter condition easily 
by the local-global principle 
(cf. Prasad-Rapinchuk \cite[Proposition
A.1]{prasad-rapinchuk:metakernel96}) which asserts that this is 
equivalent to that whenever a place $v$ of $F$ ramified in
$A$, the $F_v$-algebra $K_v:=K\otimes_F F_v$ is a
field, where $F_v$ denotes the completion of $F$ at the place $v$. 
    
It is obvious that the question {\bf (Q)} should land in the content of
linear algebra when one realizes $A=\Mat_n(\Delta)$ as the algebra of
$\Delta$-linear transformations on the $\Delta$-vector space
$\Delta^n$.    
There are several books (for example the famous book by
Jacobson~\cite{jacobson:abs2}) on linear algebra that deal with
vector spaces over division rings. However, we could not find one 
that deals with some core topics such as
eigenvalues, eigenspaces, Jordan canonical forms and 
rational canonical forms over central division algebras. 
As these are seemingly missing in the literature, we give an account 
for the theory of rational canonical forms 
over a (f.d.) central division $F$-algebra $\Delta$. 
However, our theory is explicit only for {\it separable} linear
transformations. We call a $\Delta$-linear transformation $x$ on
a (f.d.) $\Delta$-vector space $V$ {\it separable} if its
characteristic polynomial $f_x(t)$ is the product
of irreducible separable polynomials in $F[t]$. The restriction of the
separability assumption is due to the separability assumption in
a theorem of Cohen (cf. \cite[Theorem 60, p.~205]{matsumura:ca80}) 
we apply.
The main result of this Note (Theorem~\ref{4})
determines which polynomial is a characteristic polynomial.
After that we give an
explicit description of separable conjugacy classes of the multiplicative 
group $A^\times$ (Theorem~\ref{9}).

From now on, we fix a finite-dimensional central simple algebra $A$
over an arbitrary ground field $F$. Let $A=\End_{\Delta}(V)$, where $\Delta$ is
the division part of $A$ and $V$ is a right vector space over $\Delta$
of dimension $n$. Put $d:=\deg(\Delta)$. \\

\section{Minimal polynomials}
\label{sec:02}
Let $x$ be an element of $A$. The {\it minimal polynomial of $x$} is the
unique monic polynomial $m_x(t)\in F[t]$ of least degree  such that
$m_x(x)=0$ in $A$. Let $F[x]\subset A$ be the $F$-subalgebra generated
by the element $x$. One has 
$ F[x]=F[t]/(m_x(t))$ and $V$ is an
$(F[x],\Delta)$-bimodule. Since $A=\End_{\Delta}(V)$, the space $V$ is a
faithful $F[x]$-module. When the ground division algebra $\Delta$ is
equal to $F$, the theory of rational canonical forms is nothing but
the classification theorem for faithful $F[x]$-modules of
$F$-dimension $n$. Therefore, the classification theorem for
$F[x]$-faithful $(F[x],\Delta)$-bimodules of $\Delta$-dimension $n$ 
is exactly the theory of rational canonical 
forms over $\Delta$. 

\begin{lemma}\label{2}\

\begin{enumerate}

\item Let $p(t)\in F[t]$ be an irreducible polynomial of positive
  degree. Then $p(t)|m_x(t)$ if and only if $p(t)|f_x(t)$.  

\item If $m_x(t)$ is an irreducible polynomial $p(t)$, then
  $f_x(t)=p(t)^a$, where $a:=\deg(A)/\deg p(t)$. 
\end{enumerate}
   
\end{lemma}
\begin{proof}
  (1) This is a basic result in linear algebra 
   when $A$ is a matrix algebra over $F$. 
  Suppose $p(t)|f_x(t)$. Let $\alpha\in \bar F$ be a root of
  $p(t)$. Since $A\otimes \bar F$ is a matrix algebra over $\bar F$,
  the basic result shows that $(t-\alpha)|f_x(t)$ in $\bar F[t]$ if
  and only if $(t-\alpha)|m_x(t)$ in $\bar F[t]$.
  Therefore, $(t-\alpha)|m_x(t)$ in $\bar F[t]$ and hence  
  $p(t)|m_x(t)$ in $F[x]$. 

  (2) This follows from (1). \qed
\end{proof}

We write the minimal polynomial $m_x(t)=\prod_{i=1}^s p_i(t)^{e_i}$ 
into the product of irreducible polynomials, where $p_i(t)\in F[t]$ is a
monic irreducible  polynomial, $p_i\neq p_j$ if $i\neq j$, and $e_i$
is a positive integer. Put $F_i:=F[t]/(p_i(t))$ and $\wt
F_i:=F[t]/(p_i(t)^{e_i})$. By the Chinese Remainder Theorem, one has
\begin{equation}
  \label{eq:2}
   F[x]\simeq \prod_{i=1}^s \wt F_i.   
\end{equation}
Note that $F[x]$ is an Artinian ring.
The decomposition (\ref{eq:2}) is nothing but the
decomposition of $F[x]$ into the product of local Artinian rings. 
Since $\wt F_i$ is a local Artinian ring with residue field $F_i$, it
is a complete local Noetherian $F$-algebra with residue field
$F_i$. By Cohen's theorem \cite[Theorem 60, p.~205]{matsumura:ca80}, 
there is a ring monomorphism $s_i:F_i\to \wt F_i$ such 
that $\pi\circ s_i={\rm id}_{F_i}$
for all $i$, where $\pi:\wt F_i\to F_i$ is the natural projection. 
The cotangent space $\grm_{\wt F_i}/\grm^2_{\wt F_i}$ of $\wt F_i$ 
is of $F$-dimension $\deg p_i(t)$, and
hence it is a one-dimensional $F_i$-vector space. Let $\varepsilon_i\in
\wt F_i$ be a generator of the maximal ideal $\grm_{\wt F_i}$ 
of $\wt F_i$, one yields
$\wt F_i= s_i(F_i)[\varepsilon_i]/(\varepsilon_i^{e_i})$. \\

\section{Rational canonical forms}
\label{sec:03}
We now come to the $(F[x],\Delta)$-bimodule structure on $V$,
or equivalently, the (right) $\Delta\otimes_F F[x]$-module structure
on $V$. The decomposition (\ref{eq:2}) gives a decomposition of $V$
into $\Delta$-submodules $V_i$ on which the $F$-algebra $\wt F_i$ acts
faithfully:
\begin{equation}
  \label{eq:3}
   V=\bigoplus_{i=1}^s V_i, \quad \dim_{\Delta} V_i=:n_i>0, \quad
\sum_{i=1}^s n_i=n.
\end{equation}

We say an element $y$ of $A$ {\it separable} if its \ch polynomial is
the product of irreducible {\it separable} polynomials in $F[t]$. As
$A=\End_{\Delta}(V)$, a $\Delta$-linear transformation on $V$ is
called {\it separable} if it is separable as an element in $A$.  
Assume that
$x$ is 
separable. Then one can choose $s_i$ so that $F\subset s_i(F_i)$ by 
Cohen's theorem.   
Suppose that $\Delta\otimes_F s_i(F_i)=\Mat_{c_i}(\Delta_i)$,
where $c_i$ (resp. $\Delta_i$) is the capacity (resp. the division
part) of the central simple algebra 
$\Delta\otimes_F s_i(F_i)$ over $s_i(F_i)$. Each $V_i$ is a right
$\Delta\otimes_F \wt F_i$-module. One has 
\[ \Delta\otimes_F \wt F_i=[\Delta\otimes_F s_i(F_i)]\otimes_{s_i(F_i)}
\wt F_i=\Mat_{c_i}(\wt \Delta_i), \quad \wt \Delta_i:=\Delta_i
\otimes_{s_i(F_i)} \wt F_i=\Delta_i[\varepsilon_i]/(\varepsilon_i^{e_i}). \]
By the Morita equivalence, we have $V_i= W_i^{\oplus c_i}$, where
$W_i$ is a $\wt \Delta_i$-module. Since
$\Delta_i[\varepsilon_i]/(\varepsilon_i^{e_i})$ is a non-commutative PID,
there is an isomorphism of $\wt \Delta_i$-modules
\begin{equation}
  \label{eq:5}
  W_i\simeq \bigoplus_{j=1}^{t_{i}}
  \Delta_i[\varepsilon_i]/(\varepsilon_i^{m_j})  
\end{equation}
for some positive integers $1\le m_1\le \dots \le m_{t_i}\le e_i$. 
As $\wt F_i$ acts faithfully on $W_i$, one has $m_{t_i}=e_i$.    
The decomposition (\ref{eq:3}) is exactly the decomposition of $V$
into generalized eigenspaces in the classical case. 
Moreover, the decomposition (5) is exactly a 
decomposition of each generalized eigenspace into 
indecomposable invariant components. 
These give all information to make the 
rational canonical form for a suitable choice of basis as one does in
linear algebra when $\Delta=F$ (so far our theory is explicit only for
separable $\Delta$-linear transformations). 
We leave the details to the reader.  

\begin{remark}

  One needs explicit information of $\Delta, F$ and $F_i$ in order to
  compute the capacity $c_i$ and the division part $\Delta_i$ of
  $\Delta\otimes_F s_i(F_i)$ (or of $\Delta\otimes_F F_i$). 
  When $F$ is a global field, these can be computed explicitly using
  Brauer groups over local fields \cite{serre:lf} 
  and the period-index relation \cite[Theorem 32.19,
  p.~280]{reiner:mo}; see 
\cite[Section 3]{shih-yang-yu} for details. 
\end{remark}

\section{Which polynomial is characteristic?}
\label{sec:04}
Let $f(t)\in F[t]$ be a monic polynomial of degree $nd$. We will
determine when $f(t)=f_x(t)$ for some element $x\in A$. 
We say that $f(t)$ is a {\it characteristic polynomial of $A$} if
$f(t)=f_x(t)$ for some element $x\in A$. Our main result of this Note
is the following:

\begin{thm}\label{4}
  Let $f(t)\in F[t]$ be a monic polynomial of degree $nd$ and let
  $f(t)=\prod_{i=1}^s p_i(t)^{a_i}$ be the factorization into
  irreducible polynomials. Put $F_i:=F[t]/(p_i(t))$. Then $f(t)$ is a
  characteristic polynomial of $A$ if and only if for all $i=1,\dots,
  s$, one has
  \begin{itemize}
  \item [(a)] $a_i\, \deg p_i(t)=n_i
    d$ for some positive integer $n_i$, and
  \item [(b)] $[F_i:F]\mid n_i\cdot \c(\Delta\otimes_F F_i)$. 
  \end{itemize}
\end{thm}

Suppose $f(t)=f_x(t)$ for some $x\in A$. Then the minimal polynomial
$m_x(t)$ of $x$ is equal to $\prod_{i=1}^s p_i(t)^{e_i}$ 
for some positive integers $e_i$ with $e_i\le a_i$. 
Discussion in Section~\ref{sec:03} shows that there is a decomposition
of $V$ into $\Delta$-subspaces $V_i$ say of $\Delta$-dimension $n_i$
on which the $F$-algebra $\wt F_i$ acts faithfully. 
Regarding $x:V\to V$ as a $\Delta$-linear transformation, let $x_i$ 
be the restriction of $x$ on the invariant subspace $V_i$. 
Then we have
\begin{equation}
  \label{eq:55}
  f_{x_i}(t)=p_i(t)^{a_i}, \quad m_{x_i}(t)=p_i(t)^{e_i}, \quad
\text{and}\quad a_i \deg(p_i(t))=n_i d.
\end{equation}
This shows the following proposition.

\begin{prop}\label{5}
  Let $f(t)$ be as in Theorem~\ref{4}. Then $f(t)$
  is a characteristic polynomial of $A$ if and only if for
  all $i=1,\dots,s$, $a_i \deg(p_i(t))=n_i d$ for some $n_i\in \bbN$
  and the 
  polynomial $p_i(t)^{a_i}$ is a characteristic polynomial of 
  $\Mat_{n_i}(\Delta)$.  
\end{prop}

Therefore, it suffices to consider the case where $f(t)$ is a power of
an irreducible polynomial. 

\begin{lemma}\label{6}
  Let $p(t)\in F[t]$ be a monic irreducible polynomial and put
  $E:=F[t]/(p(t))$. Then a polynomial $f(t)$ of the form $p(t)^a$ 
  of degree $nd$ is
  a characteristic polynomial of $A$ if and only if there
  is an $F$-algebra embedding of $E$ in $A$. 
\end{lemma}
\begin{proof}
  Let $\bar t$ be the image of $t$ in $E$ and suppose that there is an
  $F$-algebra embedding $\rho:E\to A$. Put $x:=\rho(\bar
  t)$. Then $m_x(t)=p(t)$ and $f_x(t)=p(t)^a$ by Lemma 2 (2).
  Conversely, suppose there is an element $x\in A$ such that
  $f(t)=f_x(t)$. Let 
\[ 0=V_0\subset V_1 \subset \dots \subset V_{l-1}\subset V_l=V \]
be a maximal chain of $F[x]$-invariant $\Delta$-subspaces. The minimal
  polynomial of $x$ on each factor $V_{i}/V_{i-1}$ is equal to
  $p(t)$. Choosing an $\Delta$-linear isomorphism $\rho:\oplus_{i=1}^l
  V_{i}/V_{i-1}\simeq V$, we get an element $x^{ss}:=\rho\circ x\circ
  \rho^{-1}$ (called a
  semi-simplification of $x$) in $\End_{\Delta}(V)$ 
  whose minimal polynomial is equal to $p(t)$. 
  The map $\bar t\mapsto x^{ss}$
  gives an $F$-embedding of $E$ in $A$. \qed
\end{proof}

When $F$ is a global field and the degree $[E:F]=\deg(A)$ is maximal, 
one can use the
local-global principle to check the condition in
Lemma~\ref{6}. However, when the degree $[E:F]$ is not maximal, the
local-global principle for embedding $E$ in $A$ over $F$ fails in general;
see constructions of counterexamples in 
\cite[Sections 4 and 5]{shih-yang-yu}. 
The following lemma provides an alternative method to check this
condition. 
\begin{lemma}\label{7}
  Let $E$ be as in Lemma~\ref{6}. There is an $F$-algebra embedding
  of $E$ in $A$ if and only if $[E:F]\mid n\cdot
  \c(\Delta\otimes_F E)$.  
\end{lemma}

\begin{proof}
   This is a special case of \cite[Theorem~2.9]{yu:embed}. We provide a
   direct proof for the reader's convenience. Write $\Delta\otimes_F
   E=\Mat_{c}(\Delta')$, where $\Delta'$ is the
  division part of the central simple algebra $\Delta\otimes_F E$ over
  $E$. 
  An $F$-algebra embedding of
  $E$ into $A=\Mat_{\Delta}(V)$ exists if and only if 
  $V$ is an $(E,\Delta)$-bimodule,
  or equivalently
  a right $E\otimes_F E=\Mat_c(\Delta')$-module. 
  By the dimension
  counting, the vector space $V$ is a $\Mat_c(\Delta')$-module 
  if and only if  
  \begin{equation}
    \label{eq:42}
    \frac{\dim_F V}{c  [\Delta':F]}\in \bbN.
  \end{equation}
  Note that
  $[E:F][\Delta:F]=c^2[\Delta':F]$. 
  From this relation and that $\dim_F V=n[\Delta:F]$,  
  the condition (\ref{eq:42}) can be written as $[E:F]\mid nc$. 
  This proves the lemma. \qed
\end{proof}

By Proposition~\ref{5} and Lemmas \ref{6} and \ref{7}, 
the proof of Theorem~\ref{4} is complete. \qed

\section{Conjugacy classes}
\label{sec:05}

Let $x\in A$ be an separable element 
as in Section~\ref{sec:03}.
Suppose we have another element $x'\in A$ with the same characteristic 
polynomial $f_{x'}(t)=f_x(t)=\prod_{i=1}^s p_i(t)^{a_i}$ 
and minimal polynomial $m_{x'}(t)=m_x(t)=\prod_{i=1}^s p_i(t)^{e_i}$. 
We may identify $F[x']=F[t]/(m_x(t))=F[x]$. The
$(F[x'],\Delta)$-bimodule structure on $V$ gives a similar
decomposition into 
$\Delta$-submodules $V'_i$ as (\ref{eq:3})
\begin{equation}
  \label{eq:6}
   V=\bigoplus_{i=1}^s V'_i, \quad \dim_{\Delta} V'_i=:n'_i>0, \quad
\sum_{i=1}^s n'_i=n.
\end{equation}
Similarly, we also have 
$V'_i= (W'_i)^{\oplus c_i}$, where $W'_i$ is a $\wt
\Delta_i$-module, and 
there is an isomorphism of $\wt \Delta_i$-modules 
 \begin{equation}
  \label{eq:7}
  W'_i\simeq \bigoplus_{j=1}^{t'_{i}}
   \Delta_i[\varepsilon_i]/(\varepsilon_i^{m'_j}).  
\end{equation}
Note that $n_i=n_i'$ as  
$n_i d=a_i \deg p_i(t)=n'_i d$ (see (\ref{eq:55})). 
The elements $x$ and $x'$ are conjugate by an element in $A^\times$
if and only if the $(F[x],\Delta)$-bimodule structure and
$(F[x'],\Delta)$-bimodule structure on $V$ are equivalent. This is
equivalent to that  
\begin{equation}
  \label{eq:8}
  t_i=t_i', \quad\text{and}\quad (m_1,\dots, m_{t_i})=(m'_1,\dots,
  m'_{t_i}), \quad 
  \forall\, i=1, \dots, s. 
\end{equation}
The tuple $(m_1,\dots, m_{t_i})$ is a partition of the integer
$\dim_{\Delta_i}(W_i)$. This integer is given by the following
proposition.

\begin{prop}\label{8}
  We have $\dim_{\Delta_i} W_i=n_i c_i/\deg
  p_i(t)=a_i/\deg(\Delta_i). $
\end{prop}
\begin{proof}
  Write $\Delta\otimes s_i(F_i)=\Mat_{c_i}(\Delta_i)=
  \End_{\Delta_i}(L_i)$
  for a right $\Delta_i$-vector space $L_i$. As $\Delta\subset
  \Delta\otimes_F s_i(F_i)$ and $L_i$ is a $\Delta\otimes
  s_i(F_i)$-module, so $L_i$ is a $\Delta$-module. It follows that 
   $(\dim_{F} L_i)/[\Delta:F]\in \bbN$. Put
  $d_i=\deg(\Delta_i)$. We have $d=c_i d_i$ and 
  (using $\dim_F L_i=c_i [\Delta_i:F]$) 
  \begin{equation}
    \label{eq:54}
    c_i[\Delta_i:F]/[\Delta:F]=c_i d_i^2 \deg 
  p_i/d^2=\deg p_i/c_i \in\bbN.
  \end{equation}
Now we have
  \begin{equation}
    \label{eq:9}
    \dim_{\Delta_i}W_i=\frac{\dim_F V_i}{c_i
    [\Delta_i:F]}=\frac{[\Delta:F]\dim_\Delta
    V_i}{c_i[\Delta_i:F]}=\frac{n_i c_i}{\deg p_i}=\frac{d n_i c_i}{d
    \deg p_i}=\frac{a_i}{d_i}. 
  \end{equation}
This proves the proposition. \qed
\end{proof}

Let $S$ be the set of monic separable irreducible polynomials $p$ of
$F[t]$ with $p\neq t$. For each irreducible polynomial $p$ in $F[t]$, 
let $\Delta_p$ be the division part of the central simple algebra
$\Delta\otimes_F k(p)$ over $k(p)$, 
where $k(p)$ is the residue field at the prime
$(p)$. A partition $\lambda=(\lambda_1,\dots, \lambda_t)$ is a
non-decreasing positive integers, and we write
$|\lambda|:=\sum_{i=1}^t \lambda_i$. Put $|\lambda|=0$ if
$\lambda=\emptyset$.


\begin{thm}\label{9}
  The association from $x$ to its \ch polynomial $f_x(t)$ and the
  partitions of the integers $\dim_{\Delta_i} W_i$ by (\ref{eq:5})
  induces a bijection between the set of separable conjugacy classes
  of the 
  multiplicative group $A^\times$ and the set of partition-valued
  functions 
  $\lambda$ on $S$ such that
  \begin{equation}
    \label{eq:10}
    \sum_{p\in S} \deg(p) |\lambda(p)|
    \deg(\Delta_p)=\deg(A). 
  \end{equation}
\end{thm}
\begin{proof}
  We have shown the injectivity (see (\ref{eq:8})). 
  We show the surjectivity. Let
  $p_1,\dots, p_s$ be those with $|\lambda(p_i)|\neq 0$. Let
  $d_i:=\deg(\Delta_{p_i})$ and $a_i:=|\lambda(p_i)|
  d_i$. We need to show that the conditions (a) and (b) of
  Theorem~\ref{4} for the polynomial $\prod_{i=1}^s p_i^{a_i}$ 
  are satisfied. But these conditions are satisfied due to 
\[ n_i c_i/\deg
  p_i=a_i/d_i=|\lambda(p_i)| \in \bbN \]
and (see (\ref{eq:54}))
\[ a_i \deg 
  p_i/d=(a_i/d_i)(\deg p_i/c_i)\in \bbN.  \]  
This proves the theorem. \qed 
\end{proof}

When $A=\Delta$ is a division algebra, any minimal polynomial of
$\Delta$ is 
irreducible. The same proof presented here shows that 
the set of all conjugacy classes of $A^\times$ is parametrized by 
monic irreducible polynomials $p\neq t$ such that 
$\deg p \cdot\deg(\Delta_p)=\deg(\Delta)$.   

When ${\rm char}\, F=0$, for example if $F$ is a number field,
Theorem~\ref{9} gives a complete description of conjugacy classes
of the multiplicative group $A^\times$.  

\section*{Acknowledgements}

The author was partially supported by the grants 
NSC 97-2115-M-001-015-MY3, 100-2628-M-001-006-MY4 and AS-98-CDA-M01.
He is grateful to Ming-Chang Kang for his interest and 
helpful comments.

\end{document}